\documentclass[11pt]{article}
\usepackage{amsmath,amssymb,amsfonts,amsthm}
\usepackage{float}
\numberwithin{equation}{section}
\newtheorem{theorem}{Theorem}[section]

\newtheorem{lemma}[theorem]{Lemma}

\begin{document}
\begin{center}
{\Large{\textbf{A note on the Brush Numbers of Mycielski Graphs, $\mu(G)$}}} 
\end{center}
\vspace{0.5cm}
\large{\centerline{(Johan Kok, Susanth C, Sunny Joseph Kalayathankal)\footnote {\textbf {Affiliation of author:}\\
\noindent Johan Kok (Tshwane Metropolitan Police Department), City of Tshwane, Republic of South Africa\\
e-mail: kokkiek2@tshwane.gov.za\\ \\
\noindent Susanth C (Department of Mathematics, Vidya Academy of Science and Technology), Thalakkottukara, Thrissur-680501, Republic of India\\
e-mail: susanth\_c@yahoo.com\\ \\
\noindent Sunny Joseph Kalayathankal (Department of Mathematics, Kuriakose Elias College), Mannanam, Kottayam- 686561, Kerala, Republic of India\\
e-mail: sunnyjoseph2014@yahoo.com}}
\vspace{0.5cm}
\begin{abstract}
\noindent The concept of the brush number $b_r(G)$ was introduced for a simple connected undirected graph $G$. The concept will be applied to the Mycielskian graph $\mu(G)$ of a simple connected graph $G$ to find $b_r(\mu(G))$ in terms of an \emph{optimal orientation} of $G$. We prove a surprisingly simple general result for simple connected graphs on $n \geq 2$ vertices, namely:\\ \\
$b_r(\mu(G)) = b_r(\mu^{\rightarrow}(G))= 2\sum\limits_{i=1}^{n}d^+_{G^{\rightarrow}_{b_r(G)}}(v_i).$
\end{abstract}
\noindent {\footnotesize \textbf{AMS Classification Numbers:} 05C07, 05C12, 05C20, 05C38, 05C70} 
\section{Introduction}
\noindent For a general reference to notation and concepts of graph theory see [1]. For ease of self-containess we shall briefly introduce the concepts of \emph{brush numbers} and \emph{Mycielskian graphs}.
\subsection{The brush number of a graph $G$}
\noindent The concept of the brush number $b_r(G)$ of a simple connected graph $G$ was introduced by McKeil [3] and Messinger et. al. [5]. The problem is initially set that all edges of a simple connected undirected graph $G$ is \emph{dirty}. A finite number of brushes, $\beta_G(v) \geq 0$ is allocated to each vertex $v \in V(G).$ Sequentially any vertex which has $\beta_G(v) \geq d(v)$ brushes allocated may clean the vertex $v$ and send exactly one brush along a dirty edge and in doing so allocate an additional brush to the corresponding adjavent vertex (neighbour). The reduced graph $G' = G - vu_{\forall vu \in E(G), \beta_G(v) \geq d(v)}$ is considered for the next iterative cleaning step. Note that a neighbour of vertex $v$ in $G$ say vertex $u$, now have $\beta_{G'}(u) =\beta_G(u) + 1.$\\ \\
\noindent Clearly for any simple connected undirected graph $G$ the first step of cleaning can begin if and only if at least one vertex $v$ is allocated, $\beta_G(v)= d(v)$ brushes. The minimum number of brushes that is required to allow the first step of cleaning to begin is, $\beta_G(u) = d(u) = \delta(G).$ Note that these conditions do not guarantee that the graph will be cleaned. The conditions merely assure at least the first step of cleaning.\\ \\
\noindent If a simple connected graph $G$ is orientated to become a directed graph, brushes may only clean along an out-arc from a vertex. Cleaning may initiate from a vertex $v$\ if and only if $\beta_G(v) \geq d^+(v)$ and $d^-(v) =0.$ The order in which vertices sequentially initiate cleaning is called the \emph{cleaning sequence} in respect of the orientation $\alpha_i$. The minimum number of brushes to be allocated to clean a graph for a given orientation $\alpha_i(G)$ is denoted $b_r^{\alpha_i}$. If an orientation $\alpha_i$ renders cleaning of the graph undoable we define $b_r^{\alpha_i} = \infty.$ An orientation $\alpha_i$ for which $b_r^{\alpha_i}$ is a minimum over all possible orientations is called \emph{optimal}.\\ \\
Now, since the graph $G$ having $\epsilon(G)$ edges can have $2^{\epsilon(G)}$ orientations, the \emph{optimal orientation} is not necessary unique. Let the set $\Bbb A =\{\alpha_i|\emph{ $\alpha_i$ an orientation of G}\}.$
\begin{lemma}
For a simple connected directed graph $G$, we have that:\\ \\ $b_r(G) = min_{\emph{over all $\alpha_i \in \Bbb A$}}(\sum_{v \in V(G)}max\{0, d^+(v) - d^-(v)\}) = min_{\forall \alpha_i}b_r^{\alpha_i}.$
\end{lemma}
\begin{proof}
See [7].
\end{proof}
\noindent Although we mainly deal with simple connected graphs it is easy to see that for set of simple connected graphs $\{G_1, G_2, G_3, ..., G_n\}$ we have that, $b_r(\cup_{\forall i}G_i) = \sum\limits_{i=1}^{n}b_r(G_i).$
\subsection{Mycielskian graph $\mu(G)$ of a graph, $G$}
\noindent Mycielski [6]  introduced an interesting graph transformation in 1955. The transformation can be described as follows:\\ \\
(1) Consider any simple connected graph $G$ on $n \geq 2$ vertices labelled $v_1, v_2, v_3, ..., v_n$ and edge set $E(G)$.\\
(2) Consider the extended vertex set $V(G) \cup \{x_1, x_2, x_3, ..., x_n\}$ and add the edges $\{v_ix_j, v_jx_i|$ iff $v_iv_j \in E(G)\}.$\\
(3) Add one more vertex $w$ together with the edges $\{wx_i| \forall i\}.$\\ \\
The transformed graph \emph{(Mycielskian graph of G or Mycielski G)} denoted $\mu(G),$ is the simple connected graph with $V(\mu(G)) = V(G) \cup \{x_1, x_2, x_3, ..., x_n\} \cup \{w\}$ and $E(\mu(G)) = E(G) \cup \{v_ix_j, v_jx_i|$ iff $v_iv_j \in E(G)\} \cup \{wx_i| \forall i\}.$ 
\section{Brush Numbers of Mycielskian Graphs}
\noindent In general we have that if $\beta_{G'}(v)$ at a particular cleaning step has $\beta_{G'}(v) > d_{G'}(v)$, exactly $\beta_{G'}(v) - d_{G'}(v)$ brushes are left redundant and can clean along new edges linked to vertex $v$ if such are added through transformation of the graph $G$. It is known that for $b_r(G)$ an optimal orientation exists and brushes may only clean along out-arcs of a vertex. Construct the following directed Mycielskian graph of $G$, denoted $\mu^{\rightarrow}(G).$\\ \\
(1) Consider any simple connected graph $G$ on $n \geq 2$ vertices labelled $v_1, v_2, v_3, ..., v_n$ and edge set $E(G)$.\\
(2) Orientate $G$ corresponding to an optimal orientation associated with $b_r(G)$, denoted $G^{\rightarrow}_{b_r(G)}$.\\
(3) Consider the extended vertex set $V(G) \cup \{x_1, x_2, x_3, ..., x_n\}$ and add the arcs $\{(v_i,x_j), (v_j,x_i)|$ iff $v_iv_j \in E(G)\}.$\\
(3) Add one more vertex $w$ together with the arcs $\{(x_i, w)| \forall i\}.$\\ \\
Knowing that after adding an edge $e$ (\emph{or arc}) to a graph $G$ we have $b_r(G + e) \geq b_r(G)$ enables us to determine the brush number of the directed Mycielskian graph, $\mu^{\rightarrow}(G).$
\begin{theorem}(Tshegofatso's theorem)
For a simple connected graph $G$ on, $n \geq 2$ vertices the brush number of the Mycielskian graph of $G$ is given by:\\ \\
$b_r(\mu(G)) = b_r(\mu^{\rightarrow}(G))= 2\sum\limits_{i=1}^{n}d^+_{G^{\rightarrow}_{b_r(G)}}(v_i).$
\end{theorem}
\begin{proof}
Allocating the $b_r(G)$ brushes to the corresponding vertices of $G$ implies that the same allocations to $G^\rightarrow_{b_r(G)}$ will ensure cleaning $G^\rightarrow_{b_r(G)}$ with minimum brushes. Now consider the directed Mycielski $G$, $\mu^\rightarrow(G).$\\ \\
Consider any vertex $v \in V(G).$ Note that $d_{G^\rightarrow_{b_r(G)}}(v) = d^+_{G^\rightarrow_{b_r(G)}}(v) + d^-_{G^\rightarrow_{b_r(G)}}(v).$\\ \\
\textbf{Case 1:} Assume $d^-_{G^\rightarrow_{b_r(G)}}(v) = d^+_{G^\rightarrow_{b_r(G)}}(v).$ Clearly \emph{zero} brushes are initially allocated to \\ \\$v$ and at some iterative cleaning step exactly $d^-_{G^\rightarrow_{b_r(G)}}(v)$ brushes reaches $v$. These brushes\\ \\ will exit from $v$ along the $d^+_{G^\rightarrow_{b_r(G)}}(v)$ arcs if and only if a minimum of $d_{G^\rightarrow_{b_r(G)}}(v) = d^+_{G^\rightarrow_{b_r(G)}}(v) +\\ \\ d^-_{G^\rightarrow_{b_r(G)}}(v) = 2d^+_{G^\rightarrow_{b_r(G)}}(v)$ brushes are added to $v$ to clean the $2d^+_{G^\rightarrow_{b_r(G)}}(v)$ arcs linking $v$ with\\ \\ $2d^+_{G^\rightarrow_{b_r(G)}}(v)$ vertices $x_i \in\{x_1, x_2, x_3, ..., x_n\}.$\\ \\
So it follows that for all vertices satisfying this case we have the partial minimum sum\\ \\ of brushes, $2\sum_{v \in V(G), d^-_{G^\rightarrow_{b_r(G)}}(v) = d^+_{G^\rightarrow_{b_r(G)}}(v)}d^+_{G^\rightarrow_{b_r(G)}}(v).$\\ \\ \\
\textbf{Case 2:} Assume $d^-_{G^\rightarrow_{b_r(G)}}(v) < d^+_{G^\rightarrow_{b_r(G)}}(v).$ Clearly a minimum of  $d^+_{G^\rightarrow_{b_r(G)}}(v) - d^-_{G^\rightarrow_{b_r(G)}}(v)$\\ \\ brushes must be added to $v$ to clean all out-arcs from $v$ in $G^\rightarrow$. In addition a\\ \\ minimum of $d^-_{G^\rightarrow_{b_r(G)}}(v) + 2(d^+_{G^\rightarrow_{b_r(G)}}(v) - d^-_{G^\rightarrow_{b_r(G)}}(v))$ brushes must be allocated to $v$ to clean\\ \\ the $d^+_{G^\rightarrow_{b_r(G)}}(v) - d^-_{G^\rightarrow_{b_r(G)}}(v)$ arcs linking $v$ with vertices $x_i \in\{x_1, x_2, x_3, ..., x_n\}.$\\ \\
It follows that the minimum number of additional brushes is given by:\\ \\
$2(d^+_{G^\rightarrow_{b_r(G)}}(v) + d^-_{G^\rightarrow_{b_r(G)}}(v)) + 2(d^+_{G^\rightarrow_{b_r(G)}}(v) - d^-_{G^\rightarrow_{b_r(G)}}(v)) = 2d^+_{G^\rightarrow_{b_r(G)}}(v).$\\ \\
So it follows that for all vertices satisfying this case we have the partial minimum sum\\ \\ of brushes, $2\sum_{v \in V(G), d^-_{G^\rightarrow_{b_r(G)}}(v) < d^+_{G^\rightarrow_{b_r(G)}}(v)}d^+_{G^\rightarrow_{b_r(G)}}(v).$\\ \\ \\
\textbf{Case 3:} Assume $d^-_{G^\rightarrow_{b_r(G)}}(v) > d^+_{G^\rightarrow_{b_r(G)}}(v).$ The proof follows similar to Case 2.\footnote{The reader can formalise the proof of Case 3.}\\ \\
Since all cases have been settled and all vertices are accounted for, the result:\\ \\
$b_r(\mu(G)) = b_r(\mu^{\rightarrow}(G))= 2\sum_{v \in V(G), d^-_{G^\rightarrow_{b_r(G)}}(v) = d^+_{G^\rightarrow_{b_r(G)}}(v)}d^+_{G^\rightarrow_{b_r(G)}}(v) + \\ \\
2\sum_{v \in V(G), d^-_{G^\rightarrow_{b_r(G)}}(v) < d^+_{G^\rightarrow_{b_r(G)}}(v)}d^+_{G^\rightarrow_{b_r(G)}}(v) + 2\sum_{v \in V(G), d^-_{G^\rightarrow_{b_r(G)}}(v) > d^+_{G^\rightarrow_{b_r(G)}}(v)}d^+_{G^\rightarrow_{b_r(G)}}(v) =\\ \\ 2\sum\limits_{i=1}^{n}d^+_{G^{\rightarrow}_{b_r(G)}}(v_i),$ follows conclusively.
\end{proof}
\textbf{\emph{Open access:}} This paper is distributed under the terms of the Creative Commons Attribution License which permits any use, distribution and reproduction in any medium, provided the original author(s) and the source are credited. \\ \\ \\
References (Limited) \\ \\
$[1]$  Bondy, J.A., Murty, U.S.R., \emph {Graph Theory with Applications,} Macmillan Press, London, (1976). \\
$[2]$ Kok, J., \emph{A note on the Brush Number of Jaco Graphs, $J_n(1), n \in \Bbb N$}, arXiv: 1412.5733v1 [math.CO], 18 December 2014. \\
$[3]$ McKeil, S., \emph{Chip firing cleaning process}, M.Sc. Thesis, Dalhousie University, (2007).\\ 
$[4]$ Messinger, M. E., \emph{Methods of decontaminating a network}, Ph.D. Thesis, Dalhousie University, (2008).\\
$[5]$ Messinger, M.E., Nowakowski, R.J., Pralat, P., \emph{Cleaning a network with brushes}. Theoretical Computer Science, Vol 399, (2008), 191-205.\\
$[6]$ Mycielski, J., \emph{Sur le coloriages des graphes}, Colloq. Maths. Vol 3, (1955), 161-162.\\
$[7]$ Ta Sheng Tan,\emph{The Brush Number of the Two-Dimensional Torus,} arXiv: 1012.4634v1 [math.CO], 21 December 2010.
\end{document}